\documentclass{article}
\usepackage{amsmath}
\usepackage{amssymb}
\usepackage{amsfonts}
\usepackage{amsthm}
\usepackage{cancel}
\usepackage{physics}
\usepackage{xcolor}
\usepackage[a4paper,top=2.5cm,bottom=3cm,left=2.5cm,right=2.5cm,marginparwidth=2cm]{geometry}
\usepackage{tikz}
\usetikzlibrary{arrows.meta,positioning,decorations.pathmorphing,patterns,external}
\usepackage{float}
\usepackage{verbatim}
\usepackage{subcaption}
\usepackage{pgfplots}
\pgfplotsset{width=12cm,compat=1.17}
\usepgfplotslibrary{fillbetween}

\theoremstyle{definition} 
\newtheorem{lemma}{Lemma}[section]
\newtheorem{theorem}{Theorem}[section]
\newtheorem{prop}{Proposition}[section]

\newtheorem{definition}{Definition}[section]

\numberwithin{equation}{section}

\newcommand{\SL}{\text{SL}_2(\mathbb{R})}
\newcommand{\e}[1]{\mathbf{e}_{#1}}
\newcommand{\T}[1]{\begin{bmatrix} 0 & -1 \\ 1 & #1 \end{bmatrix}}

\newcommand{\Tc}{\mathcal{T}}
\newcommand{\M}{\mathcal{M}}

\DeclareMathOperator*{\arginf}{arg\,inf}

\title{Hierarchical band gaps in complex periodic systems}
\author{Lucas Dunckley\thanks{Department of Mathematics, Imperial College London, 180 Queen's Gate, London SW7 2AZ, United Kingdom.} \and Bryn Davies\footnotemark[1]}
\date{}

\begin{document}
\maketitle
\begin{abstract}
    Complex periodic structures inherit spectral properties from the constituent parts of their unit cells, chiefly their spectral band gaps. Exploiting this intuitive principle, which is made precise in this work, means spectral features of periodic systems with very large unit cells can be predicted without numerical simulation. We study a class of difference equations with periodic coefficients and show that they inherit spectral gaps from their constituent elements. This result shows that if a frequency falls in a band gap for every constituent element then it must be in a band gap for the combined complex periodic structure. This theory and its instantaneous utility is demonstrated in a series of vibro-acoustic and mechanical examples.
\end{abstract}

\section{Introduction}

A wealth of advanced technologies depend on the varied (and sometimes surprising) properties exhibited by waves propagating through systems with periodically varied material parameters. These technologies include photonic and phononic crystals capable of precise wave control \cite{joannopoulos1995molding} and metamaterials that display exotic effective properties \cite{smith2004metamaterials}. Such metamaterials have been realised in almost every corner of wave physics and can perform impressive feats such as negative refraction and cloaking \cite{milton2006cloaking, craster2012acoustic}. Partly as a result of this ongoing technological revolution, there is a wealth of literature devoted to studying the spectra of the associated operators (either differential or discrete) with periodic coefficients.

As photonic crystals and metamaterials grow in complexity, there is a need to understand the spectra of periodic systems with increasingly complex unit cells. There has even been some interest in unit cells with fractal geometries \cite{bao2007surface, wen2005resonant}, which presents its own specific theoretical challenges. Under reasonable assumptions (and excluding fractals), the complexity of the unit cell typically does not cause any issues for the standard theoretical tools such as Floquet-Bloch analysis \cite{kuchment2016} or two-scale homogenisation \cite{allaire1992homogenization}. The challenge, however, when faced with a periodic system with a highly complex unit cell, is to develop some intuition for the main features of its spectrum.

The main spectral feature that it is useful to be able to predict is the occurrence of band gaps. These are ranges of frequencies at which waves are unable to propagate through the periodic system. Instead, the amplitude decays exponentially as a function of position and energy is typically reflected back in the direction of incidence. There are a variety of approaches for predicting if a system will support band gaps, without computing the spectrum in full. These are typically based on examining the symmetries of the material and can be framed in terms of either point symmetry groups \cite{makwana2018geometrically}, frieze groups \cite{zhang2019symmetry} or topological invariants \cite{bansil2016colloquium}. 

In this work, we characterise the spectra of complex periodic materials by decomposing the unit cell into simpler component elements. Equivalently, we are interested in complex periodic materials that are formed by combining the unit cells of simpler periodic materials. This is depicted in Figure~\ref{fig:Hero}. The key question is whether the complex material inherits the band gaps of the simpler materials that make up its unit cell. We consider systems with two degrees of freedom, which could include a broad variety of mechanical, acoustic or suitably polarized electromagnetic systems (such that the Helmholtz equation captures the wave dynamics). In such systems, wave propagation can be described by a two-by-two transfer matrix \cite{markos2008wave} and band gaps are precisely the values at which the trace of the transfer matrix is larger than 2 in absolute value. As a result, the question at the heart of this study is whether the set
\begin{equation}
    \M:=\left\{ M\in\SL : |\tr(M)|>2 \right\},
\end{equation}
is closed under matrix multiplication. Here, $\SL$ is the special linear group of $2\times 2$ matrices with real entries and determinant 1. It is clear that this is not true in general. To see this, take 
two arbitrary matrices $A,B\in\M$ such that there exists constants $ \alpha\neq\pm1,a,b,c,d\in\mathbb{R}$ and some $P\in\text{GL}_2(\mathbb{R})$ such that
\begin{equation}
A \sim P^{-1}AP = \begin{bmatrix}
    \alpha & 0 \\ 0 & \frac{1}{\alpha}
\end{bmatrix}\quad\text{and}\quad B \sim P^{-1}BP = \begin{bmatrix}
    a & b \\ c & d
\end{bmatrix},
\end{equation}
where $\sim$ is used to denote matrix similarity. Then, it holds that 
$$\tr(AB) = \tr(P^{-1}ABP) = \alpha a + \frac{d}{\alpha}, $$
from which it is clear that there are uncountably many matrices $B$ with $|a+d|>2$ such that $|\alpha a + d/\alpha|\ge2.$ Therefore, we must constrain the problem and consider a subset of $\M$.

\begin{figure}[!h]
    \centering
    \includegraphics{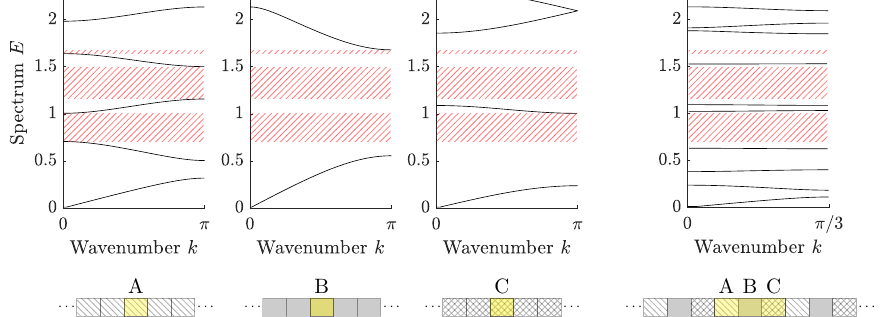}
    \caption{This work develops a theory for hierarchical band gaps in complex periodic systems. These are band gaps in complex periodic systems which are common with the band gaps of all the constituent simple systems. For example, the spectra of three materials A, B and C are shown here, with their common band gaps highlighted and inherited by the combined periodic material with unit cell ABC. This presents a simple way to predict some, but not all, of the band gaps in complex periodic systems.}
    \label{fig:Hero}
\end{figure}

In this work, we prove a partial result about the closure of subsets of $\M$ that is sufficient for yielding insight on applications. Our result concerns products of matrices belonging to the subset $\Tc\subset\M$, given by
\begin{equation}
    \Tc:=\left\{ \T{t} : |t|>2 \right\}.
\end{equation}
We show that any finite product of elements of $\Tc$ will be an element of $\M$. This result can be used to predict the spectra of a class of one-dimensional difference equations with on-site modulation. In particular, it shows that if the intersection of the band gaps of the constituent simple (1-periodic) materials is non-empty, then that interval will be within a band gap of the complex periodic material (as depicted in Figure~\ref{fig:Hero}). We will refer to this band gap that has been inherited from the constituent elements as a \emph{hierarchical band gap}.

This paper will begin by laying out the problem setting in section~\ref{sec:problem}, including specifying how the set $\Tc$ is related to the Floquet-Bloch spectrum of a class of difference equations. Section~\ref{sec:theory} contains the main theoretical results. Finally, in section~\ref{sec:examples}, we will consider some specific examples of mechanical and acoustic systems to demonstrate the value of our theory for predicting spectral features.

\section{Problem setting} \label{sec:problem}

In this work, we consider a canonical problem given by the difference equation
\begin{equation} \label{eq:difference}
    g(i+1) + g(i-1) + V(i)g(i) = F(i,\omega)g(i),
\end{equation}
where $V:\mathbb{Z}\to\mathbb{R}$ and $F:\mathbb{Z}\times\mathbb{R}\to\mathbb{R}$ are given functions and $\omega\in\mathbb{R}$ is a scalar eigenvalue that describes the frequency of waves propagating through the system. The spectral problem in question is to find all the values of $\omega\in\mathbb{R}$ such that \eqref{eq:difference} has a non-zero solution $g:\mathbb{Z}\to\mathbb{R}$. Difference equations of this type arise naturally in many different settings, such as tight-binding models for electron dynamics, one-dimensional wave transmission problems or in systems of coupled mechanical oscillators \cite{markos2008wave}. We will explore some examples in more detail in section~\ref{sec:examples}. The difference equation \eqref{eq:difference} is a typical example of a discrete Hamiltonian with nearest-neighbour coupling and on-site modulation. 

We are interested in the case where $V:\mathbb{Z}\to\mathbb{R}$ is periodic and $F:\mathbb{Z}\times\mathbb{R}\to\mathbb{R}$ is periodic in its first variable. That is, there is some $n>0$ such that $V(i+n)=V(i)$ and $F(i+n,\cdot)=F(i,\cdot)$ for all $n\in\mathbb{Z}$. In this case, the spectrum is known to consist of a countable number of continuous spectral bands. These bands can be resolved using so-called Floquet-Bloch analysis, of which we now summarise the main ideas (for a more detailed overview, see \emph{e.g.} \cite{kuchment2016, brillouin1953wave}). 

The difference equation \eqref{eq:difference} can be rewritten as
\begin{align} 
    \begin{bmatrix}
        g(i) \\ g(i+1) 
    \end{bmatrix} &= -\T{V(i)-F(i,\omega)} \begin{bmatrix}
        g(i-1) \\ g(i)
    \end{bmatrix}. \label{eq:system_1}
\end{align}
Letting $\mathbf{u_{i}} = [g(i-1),g(i)]^\dagger$, where $^\dagger$ is used to denote the transpose, we can rewrite \eqref{eq:system_1} as
\begin{equation} \label{eq:matrix_1}
    \mathbf{u}_{i+1} = -T_i \mathbf{u}_i \quad\text{where}\quad T_i = \T{V(i)-F(i,\omega)}.
\end{equation}
Further, across the unit cell of length $n$ it holds that
\begin{equation} \label{eq:unitcell}
    \mathbf{u}_{n+1} = (-1)^n \left( \prod_{i=1}^n T_i \right) \mathbf{u}_1.
\end{equation}
The key insight of Floquet-Bloch analysis is that the spectrum can be decomposed into modes satisfying
\begin{equation} \label{eq:FB}
\mathbf{u}_{n+1} =e^{\mathrm{i}nk} \mathbf{u}_{1}
\end{equation}
for some $k\in\mathbb{R}$. For each $k$, there exists a discrete set of countably many eigenvalues and the spectral bands of the periodic problem are recovered by taking the union over all possible $k$. Clearly, it is sufficient to take this union over $k\in[-\pi/n,\pi/n)$. Equating \eqref{eq:unitcell} with \eqref{eq:FB} and taking the determinant gives
$$ \det\left((-1)^n\tr\left( \prod_{i=1}^n T_i \right) - e^{\mathrm{i}nk}I\right) = 0.$$
Using the fact that $\det T_i \equiv 1$, this reduces to the well-known dispersion relation
\begin{equation} \label{dispersion}
    2\cos(nk) = (-1)^n\tr\left( \prod_{i=1}^n T_i \right).
\end{equation}
It is clear that \eqref{dispersion} has a real solution for $k$ if and only if the right-hand side falls in $[-2,2]$. Hence, the condition for an eigenvalue $\omega$ to belong to a spectral band of the periodic system is that the associated transfer matrix has trace smaller than 2 in absolute value (with the associated matrix being defined by \eqref{eq:unitcell}, where it should be noted that the matrices in the product depend on $\omega$). 

The intervals of frequencies between spectral bands are known as \emph{band gaps} and are significant as waves at these frequencies are unable to propagate through the material. The aim of this work is to characterise how complex systems inherit band gaps from their constituent elements. Hence, we will show that the product of matrices of the form \eqref{eq:matrix_1} with $|t_i|=|V(i)-F(i,\omega)|>2$ must also have trace larger than 2. We will say that any $\omega$ which satisfies this condition belongs to a \emph{hierarchical band gap} of the complex periodic material.
\begin{definition}
    Given $n\in\mathbb{N}$ and functions $t_i:\mathbb{R}\to\mathbb{R}$ for $1\le i\le n$, a \emph{hierarchical band gap} is some interval $(a,b)\subset\mathbb{R}$ such that for $\omega\in(a,b)$ it holds that both
\begin{equation}\label{eq:hierarchical_def}
    \left|\tr \T{t_i}(\omega)\right| > 2 \quad\text{for all } 1\le i\le n
\end{equation}
and
\begin{equation}\label{eq:hierarchical_def2}
    \left|\tr \left( \,\prod_{i=1}^n \T{t_i(\omega)} \right) \right| > 2.
\end{equation}
\end{definition}

The theory developed in section~\ref{sec:theory} will show that, in fact, \eqref{eq:hierarchical_def} implies \eqref{eq:hierarchical_def2} (\emph{i.e.} the product of $n$ matrices belonging to $\Tc$ must belong to $\M$). This demonstrates how band gaps persist in systems with increasingly complex arrangements. Some examples of such arrangements will be explored in section~\ref{sec:examples}.

\section{Main results} \label{sec:theory}

To see why products of matrices belonging to $\Tc$ must themselves belong to $\M$, we exploit the patterns that occur when multiplying matrices of the form \eqref{eq:matrix_1}. Before stating any results we introduce some notation to simplify statements:
\begin{equation}
    \Pi_k^l := \prod_{i=k}^l T_i, \label{eq:pf_1}
\end{equation}
where $T_i$ is some arbitrary matrix of the form
\begin{equation}
    T_i = \T{t_i}.
\end{equation}
We will use the notation $\e{1}$ and $\e{2}$ to denote the standard basis elements in $\mathbb{R}^2$:
\begin{equation}
    \e{1}=\begin{bmatrix} 1 \\ 0\end{bmatrix} \quad\text{and}\quad \e{2}=\begin{bmatrix} 0 \\ 1\end{bmatrix}.
\end{equation}

It will be useful to begin by establishing that the result holds for edge cases (\emph{i.e.} $t_i\equiv2$ and $t_i\equiv-2$), before moving on to the general case.
\begin{prop} \label{formula_2}
    For any $n\in\mathbb{N}$, it holds that
    \begin{equation*}
        \T{2}^n = \T{2} + (n-1)\begin{bmatrix}
        -1 & -1 \\ 1 & 1 
    \end{bmatrix},
    \end{equation*} 
    the trace of which is always equal to 2.
\end{prop}
\begin{proof}
    \begin{align*}
        \T{2}^{n} &= \T{2} \cdot \left\{ \T{2} + (n-1)\begin{bmatrix}
            -1 & -1 \\ 1 & 1 
        \end{bmatrix} \right\} \\
        &= \begin{bmatrix}
            -1 & -2 \\ 2 & 3 
        \end{bmatrix} + (n-1) \begin{bmatrix}
            -1 & -1 \\ 1 & 1 
        \end{bmatrix} \\
        &= \T{2} + n \begin{bmatrix}
            -1 & -1 \\ 1 & 1 
        \end{bmatrix} 
    \end{align*}
\end{proof}

\begin{prop}\label{formula_-2}
    For any $n\in\mathbb{N}$, it holds that
    \begin{equation*}
        \T{-2}^n = (-1)^n \left\{ (n-1) \begin{bmatrix}
        -1 & 1 \\ -1 & 1 
        \end{bmatrix} - \T{-2} \right\},
    \end{equation*}
    the trace of which is always equal to $\pm2$.
\end{prop}
\begin{proof}
    \begin{align*}
        \T{-2}^{n} &= (-1)^n \T{-2} \left\{ (n-1) \begin{bmatrix}
            -1 & 1 \\ -1 & 1 
        \end{bmatrix} - \T{-2} \right\}\\
        &= (-1)^{n} n \begin{bmatrix} 
            -1 & 1 \\ -1 & 1 
        \end{bmatrix} + (-1)^{n+2} \T{-2}\\
        &= (-1)^{n} \left\{ n\begin{bmatrix} 
            -1 & 1 \\ -1 & 1
        \end{bmatrix} - \T{-2} \right\}
    \end{align*}
\end{proof}


We are now ready to handle the general case of $|t_i|\geq 2$.

\begin{lemma}\label{lem:proof2_1}
    Let $t_1,...,t_n\in\mathbb{R}$ be such that  $\min_{1\leq i\leq n}|t_i| \ge 2$, then 
    $$ \left| \e{2}^\dagger \Pi_1^n \e{2} \right| > \left| \e{1}^\dagger \Pi_1^n \e{2} \right| + 1 \quad\text{and}\quad \left| \e{2}^\dagger \Pi_1^n \e{2} \right| > \left| \e{2}^\dagger \Pi_1^n \e{1} \right| + 1. $$
\end{lemma}
\begin{proof}
    This follows by induction. For $n=1$, it holds that
    \begin{equation}
        \Pi_1^1 \equiv \T{t} \quad\text{for some }t \in (-\infty,-2]\cup[2,\infty)
    \end{equation}
    so the result is trivially true.

    To prove the first of the two statements, consider splitting the product as
    \begin{equation} \label{eq:split}
        \Pi_1^n = \T{t_1} \cdot \Pi_2^n,
    \end{equation}
    from which it is clear that 
    \begin{gather}
        \e{1}^\dagger \Pi_1^n \e{2} = -\e{2}^\dagger \Pi_2^n \e{2}, \label{eq:proof2_1_1} \\
        \e{2}^\dagger \Pi_1^n \e{2} = \e{1}^\dagger \Pi_2^n \e{2} + t_1\e{2}^\dagger \Pi_2^n \e{2}. \label{eq:proof2_1_2}
    \end{gather}
    Where $\dagger$ denotes transpose. Taking absolute values gives
    \begin{equation}
         \left| \e{2}^\dagger \Pi_1^n \e{2} \right| \ge \left|\left\{ \left| t_1\e{2}^\dagger \Pi_2^n \e{2} \right| - \left| \e{1}^\dagger \Pi_2^n \e{2} \right| \right\}\right|,
    \end{equation}
    which can be simplified to see that
    \begin{align}
         \left| \e{2}^\dagger \Pi_1^n \e{2} \right| &\ge
         \left| t_1\e{2}^\dagger \Pi_2^n \e{2} \right| - \left| \e{1}^\dagger \Pi_2^n \e{2} \right| \nonumber \\
            &= \left(|t_1|-1\right)\left| \e{2}^\dagger \Pi_2^n \e{2} \right| + \left| \e{2} \Pi_2^n \e{2} \right| - \left| \e{1}^\dagger \Pi_2^n \e{2} \right| \nonumber\\
            &> \left(|t_1|-1\right) \left| \e{2}^\dagger \Pi_2^n \e{2} \right| + 1 \nonumber \\
            &> \left| \e{2}^\dagger \Pi_2^n \e{2} \right| + 1, \nonumber 
    \end{align}
    where the penultimate inequality follows from the inductive hypothesis. Using \eqref{eq:proof2_1_1}, this becomes  
        \begin{align}
            \left| \e{2}^\dagger \Pi_1^n \e{2} \right| &> \left| \e{1}^\dagger \Pi_1^n \e{2} \right| + 1.
        \end{align}

    The second statement can be shown using a similar approach. We use the formula
    \begin{equation}
        \Pi_1^n = \Pi_1^{n-1} \cdot \T{t_n},
    \end{equation}
    to see that
    \begin{gather}
        \e{2}^\dagger \Pi_1^n \e{1} = \e{2}^\dagger \Pi_1^{n-1} \e{2}, \label{eq:proof2_2_1}\\
        \e{2}^\dagger \Pi_1^n \e{2} = -\e{2}^\dagger \Pi_1^{n-1} \e{1} + t_n \e{2}^\dagger \Pi_1^{n-1} \e{2} \label{eq:proof2_2_2}.
    \end{gather}
    Therefore, we have that
    \begin{align*}
        \left| \e{2}^\dagger \Pi_1^n \e{2} \right| &\ge \left|\left\{ \left| t_n \e{2}^\dagger \Pi_1^{n-1} \e{2} \right| - \left| \e{2}^\dagger \Pi_1^{n-1} \e{1} \right| \right\} \right| \\
        &= \left| t_n \e{2}^\dagger \Pi_1^{n-1} \e{2} \right| - \left| \e{2}^\dagger \Pi_1^{n-1} \e{1} \right| \nonumber \\
        &= \left( |t_n| - 1 \right) \left| \e{2}^\dagger \Pi_1^{n-1} \e{2} \right| + \left| \e{2}^\dagger \Pi_1^{n-1} \e{2} \right| - \left| \e{2}^\dagger \Pi_1^{n-1} \e{1} \right| \nonumber \\
        &> \left| \e{2}^\dagger \Pi_1^{n-1} \e{2} \right| + 1 \nonumber \\
        & = \left| \e{2}^\dagger \Pi_1^n \e{1} \right| + 1,
    \end{align*}
    which proves the second condition.
\end{proof}

\begin{theorem}\label{thm:general_result}
    If for some arbitrary $n\in\mathbb{N}$, for all $i \le n$, $|t_i|>2$, then 
    \begin{equation*} \left| \tr\left\{ \prod_1^n \T{t_i} \right\} \right| > 2. \end{equation*}
\end{theorem}
\begin{proof}
    Re-using the simple formula \eqref{eq:split}, we can see that
    \begin{equation}
        \tr \Pi_1^{n} = \tr\left\{ \T{t_1} \Pi_2^{n} \right\} 
        = -\e{2}^\dagger \Pi_2^{n} \e{1} + \e{1}^\dagger \Pi_2^{n} \e{2} + t_1 \e{2}^\dagger \Pi_2^{n} \e{2}.
    \end{equation}
    Taking absolute values and applying Lemma~\ref{lem:proof2_1} gives that
    \begin{align*}
        \left| \tr \Pi_1^{n} \right| &\ge \left|\left\{ \left| t_1 \e{2}^\dagger \Pi_2^{n} \e{2} \right| - \left| \e{1}^\dagger \Pi_2^{n} \e{2} \right| - \left| \e{2}^\dagger \Pi_2^{n} \e{2} \right| \right\}\right| \\
        &= \left| t_1 \e{2}^\dagger \Pi_2^{n} \e{2} \right| - \left| \e{1}^\dagger \Pi_2^{n} \e{2} \right| - \left| \e{2}^\dagger \Pi_2^{n} \e{2} \right| \\
        &> \left\{ |t_1|-2 \right\} \left| \e{2}^\dagger \Pi_2^{n} \e{2} \right| + 2 \\
        &> 2. \qedhere
    \end{align*}
\end{proof}

Theorem~\ref{thm:general_result} is the main theoretical result of this work. It shows that the product of a finite number of matrices in $\Tc$ will be an element of $\M$. In other words, given a complex periodic material that has a unit cell with length $n$, if there is an eigenvalue that is in the band gaps of all of the $n$ periodic materials associated to the $n$ individual sites, then that eigenvalue must be in a band gap of the complex periodic material. Since the intersection of a finite number of open intervals is itself an open interval, any such eigenvalues must appear in continuous bands. These are the sought hierarchical band gaps.

It should be emphasised that Theorem~\ref{thm:general_result} describes a sufficient but not necessary condition for the complex periodic material to have a band gap. Indeed, it is easy to construct examples where the trace of the product has absolute value greater than two in spite of the fact that there is at least one $i$ for which $|t_i|\leq 2$ (typically, many of the other values would be $|t_i|\gg 2$ to compensate). Predicting the spectrum in full needs more knowledge of the specific configuration of the system than is required for  Theorem~\ref{thm:general_result}. Nevertheless, the simple condition derived in Theorem~\ref{thm:general_result} is sufficient for predicting the main band gaps in complex systems where there is some overlap between the spectra of the individual constituent elements.

\section{Examples} \label{sec:examples}

The difference equation \eqref{eq:difference} is a canonical model for wave propagation in one-dimensional systems. Some specific examples of appropriate systems are coupled masses and springs, coupled pendulums and locally resonant phononic crystals. These examples are surveyed below, to demonstrate the utility of Theorem~\ref{thm:general_result}.

\subsection{Masses and springs} \label{sec:massspring}

The simplest canonical model for wave propagation in a lattice structure is that of an array of masses coupled by springs, as depicted in figure~\ref{fig:masses_1}. These models have been studied as early as the 17\textsuperscript{th} century (such as by Newton in his Principia of 1686) and remain one of the primary toy models for exploring wave dynamics in both quantum and classical settings \cite{brillouin1953wave}.

\begin{figure}[H]
    \centering
    \includegraphics{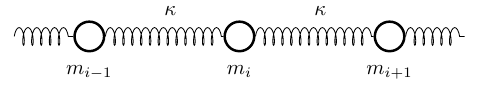}
    \caption{The canonical difference equation studied here can be used to model a system of masses (with modulated masses $m_i$) coupled by identical springs (with spring constant $\kappa$).}
    \label{fig:masses_1}
\end{figure}

We consider the propagation of time-harmonic waves with frequency $\omega$ through a chain of masses with mass $m_i$ that are connected by identical springs with stiffness $\kappa$. Then, the dynamics are described by the difference equation \cite{makwana2013localised}
\begin{equation}\label{eq:springs_1}
    u_{i+1}+u_{i-1}-2u_i=-\frac{\omega^2m_i}{\kappa}u_i.
\end{equation}
This is equivalent to \eqref{eq:difference} with $V(i)-E=\omega^2m_i/\kappa-2$.
This difference equation can be rewritten as
\begin{equation}\label{eq:springs_2}
    \begin{bmatrix}
        u_i \\ u_{i+1} 
    \end{bmatrix} = -\begin{bmatrix}
        0 & -1 \\ 1 & \frac{\omega^2 m_i}{\kappa} -2
    \end{bmatrix} \begin{bmatrix}
        u_{i-1} \\ u_i 
    \end{bmatrix}
\end{equation}
where the matrix is an element of $\Tc$ if and only if
\begin{equation}
    \left| \frac{\omega^2 m_i}{\kappa} - 2\right|>2.
\end{equation}
It is immediate that for any simple singly periodic system with this matrix there is one high-frequency band gap that occurs when $\omega^2 > \frac{4\kappa}{m_i}$. Therefore, it is expected that a complex periodic arrangement of masses and springs will also inherit a high-frequency band gap. 

An example is shown in Figure~\ref{fig:masses_2}, where we study a system of 5 masses with masses $m_1 = 0.6$, $m_2 = 1$, $m_3 = 0.8$, $m_4 = 0.5$ and $m_5 = 1$, coupled with identical springs with spring constant $\kappa=0.75$. The top five rows show the pattern of pass and stop bands for each of the five periodic systems corresponding to the single mass $m_i$ repeated periodically. The bottom row shows the pass and stop bands for the complex system obtained by combining the five masses. The hierarchical band gap, as predicted by our theory, is highlighted by the hatched region. The bottom of the (high-frequency) hierarchical band gap is highlighted by the dashed line, which occurs at $4\kappa/m_N$ where $N = \arginf_{i\le n} m_i$.

\begin{figure}
    \centering
    \includegraphics{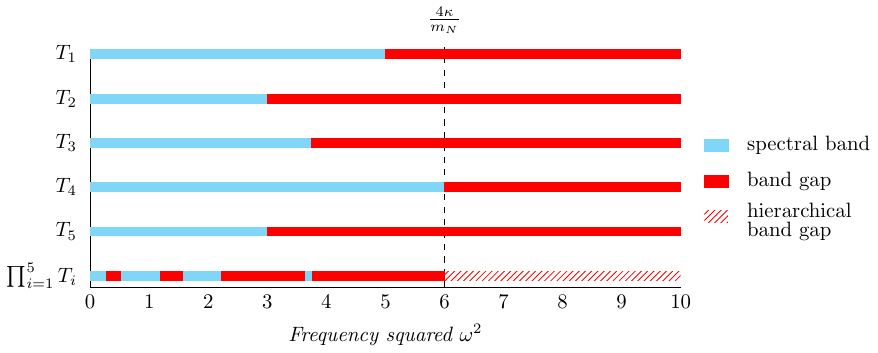}
    \caption{Hierarchical band gaps of a system of coupled masses and springs, where the complex system inherits the common high-frequency band gap of its constituent elements. A simple periodic structure is considered, composed of five constituent elements.}
    \label{fig:masses_2}
\end{figure}

The high-frequency band gap of a modulated system of masses and springs is widely observed in the physical literature. Modifying the asymptotic arguments from \cite[Theorem~4.1]{davies2023super}, for example, it is easy to show that any periodic mass-spring system will exhibit a band gap at sufficiently high frequencies. The theory developed in the present work extends this result by giving a quantitative estimate for the minimum of this high-frequency band gap, showing that it is at most the minimum of the intersection of band gaps of the constituent elements.

It is important to note that our theory for hierarchical band gaps does not contain any information about the relative positioning of the masses within the unit cell of the complex material. As a result, the hierarchical band gap will exist for any arrangement of the constituent masses (including adding many repetitions of some of them). This is one reason why we cannot expect, in general, the hierarchical band gap to provide tight bounds on the band gaps of the complex periodic system.

\subsection{Coupled pendulums}

The addition of local interactions is an immediate and important development of the mass-spring model. This can be achieved, for example, by adding pendulums to the masses, as depicted in Figure~\ref{fig:pendulums}. This is a canonical model for wave propagation in locally resonant systems and mimics the key properties of the tight-binding models used in quantum physics \cite{economou2006quantum}.

\subsubsection{Homogeneous Pendulums}

Consider the system presented in Figure~\ref{fig:pendulums} with identical springs of stiffness $\kappa$ and identical pendulums of length $l$ attached to masses, which are allowed to vary. We will address the case of non-identical pendulums in section~\ref{sec:pendvary}.

\begin{figure}[H]
    \centering
    \includegraphics{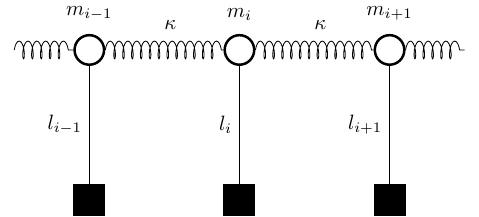}
    \caption{A system of coupled pendulums (which are coupled by indentical springs) can be described by difference equations of the form considered in this work.}
    \label{fig:pendulums}
\end{figure}

Let $u_n$ be the one-dimensional displacement of the sphere with mass $m_n$, then resolving the forces due to both the springs and the pendulum we arrive at the difference equation \cite{markos2008wave}
\begin{equation}\label{eq:pendulums_1}
    \left( \omega_i^2 + \frac{1}{m_i}(2\kappa - \omega^2)\right) u_i = \frac{1}{m_i}(\kappa u_{i-1} + \kappa u_{i+1}).
\end{equation}
Here, $\kappa$ is the spring constant, $\omega$ is the frequency of the wave in the system and $\omega_i$ is the eigenfrequency of the pendulum with index $i$. This eigenfrequency is given by $\omega_i^2= \frac{g}{l}$ where $l$ is the length of the pendulum and $g$ is the acceleration due to gravity. For simplicity, in the case of identical pendulums we denote this as $\omega_0$ to remove dependence on $i$. We can rewrite \eqref{eq:pendulums_1} as the linear system
\begin{equation}\label{eq:pendulums_2}
    \begin{bmatrix}
        u_i \\ u_{i+1} 
    \end{bmatrix} = \begin{bmatrix}
        0 & 1 \\ -1 & 2 - \frac{\omega^2m_i}{\kappa} + \frac{\omega_0^2 m_i}{\kappa} 
    \end{bmatrix} \begin{bmatrix}
        u_{i-1} \\ u_i 
    \end{bmatrix}
\end{equation}
This matrix has trace larger than 2 if
$$ \left|2 - \frac{\omega^2m_i}{\kappa} + \frac{\omega_0^2 m_i}{\kappa}\right|>2, $$
which is equivalent to 
\begin{equation}\label{eq:pendulums_3}
    \omega^2 < \omega_0^2 \quad \text{or} \quad \omega^2 > \frac{4\kappa}{m_i} + \omega_0^2.
\end{equation}
We can see that the introduction of the pendulums has introduced a low-frequency band gap that was not present in the mass-spring system considered in section~\ref{sec:massspring}.

\begin{figure}[H]
    \centering
    \includegraphics{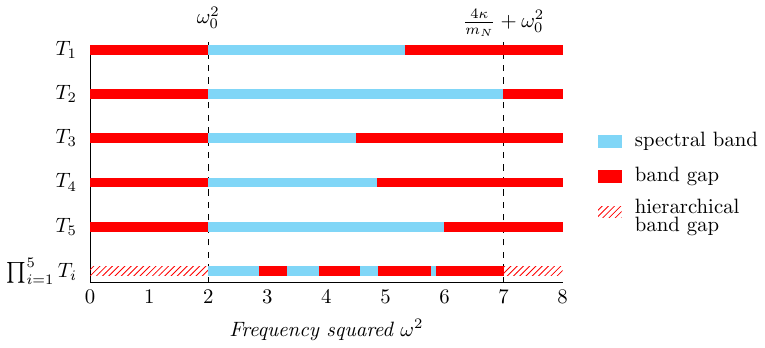}
    \caption{Hierarchical band gaps of a system of coupled pendulums, with a simple periodic unit cell consisting of five elements. The complex pendulum periodic structure inherits both high- and low-frequency hierarchical band gaps from the constituent systems.}
    \label{fig:pendulums_2}
\end{figure}

An example of the band gaps corresponding to a system of five masses and pendulums is shown in Figure~\ref{fig:pendulums_2}. The masses have been chosen to be such that $m_1 = 0.6$, $m_2 = 0.4$, $m_3 = 0.8$, $m_4 = 0.7$ and $m_5 = 0$, with $\kappa=0.5$ and $\omega_0^2 = 2$. This gives a range of different band gaps (with common lower limits due to the identical pendulums). The corresponding complex periodic system has hierarchical band gaps at either ends of its spectrum, as the complex system inherits both the low- and high-frequency band gaps from the constituent elements. Of course, the behaviour of the spectrum in the frequency range $\omega_0^2 < \omega < w_0^2 +4m_N\kappa$ is undetermined, a reminder of the sufficient but not necessary nature of the hierarchical band gap condition.

\subsubsection{Modulated Pendulums} \label{sec:pendvary}

More intricate spectra can be obtained by varying the lengths of the pendulums. In which case, the eigenfrequency $\omega_i$ in \eqref{eq:pendulums_1} varies and we obtain the linear system 
\begin{equation}\label{eq:pendulums_4}
    \begin{bmatrix}
        u_i \\ u_{i+1} 
    \end{bmatrix} = \begin{bmatrix}
        0 & 1 \\ -1 & 2 - \frac{\omega^2m_i}{\kappa} + \frac{\omega_n^2 m_i}{\kappa} 
    \end{bmatrix} \begin{bmatrix}
        u_{i-1} \\ u_i 
    \end{bmatrix}.
\end{equation}
Similar to \eqref{eq:pendulums_3}, we find the condition for each individual mass-pendulum system to support a band gap as
\begin{equation}\label{eq:pendulums_5}
    \omega^2 < \omega_i^2 \quad \text{or} \quad \omega^2 > \frac{4\kappa}{m_i} + \omega_i^2.
\end{equation}
Then, under the same assumptions as Figure~\ref{fig:pendulums_2}, we build a system with varying lengths.  In particular, we suppose that the pendulums have one of two alternating lengths: $l_{2n} = 20$ and $l_{2n+1}=5$.
If, for simplicity, we take $g=10$, then this leads to $\omega_{2n}^2 = 0.5$ and $\omega_{2n+1}^2 = 2$. Choosing the spring stiffness and masses as $\kappa = 0.5$, $m_1 = 1.2$, $m_2 = 2$, $m_3 = 1$, $m_4 = 2.2$ and $m_5 = 1.2$, we reach the pattern of spectral bands shown in Figure~\ref{fig:pendulums_3}.

\begin{figure}[H]
    \centering
    \includegraphics{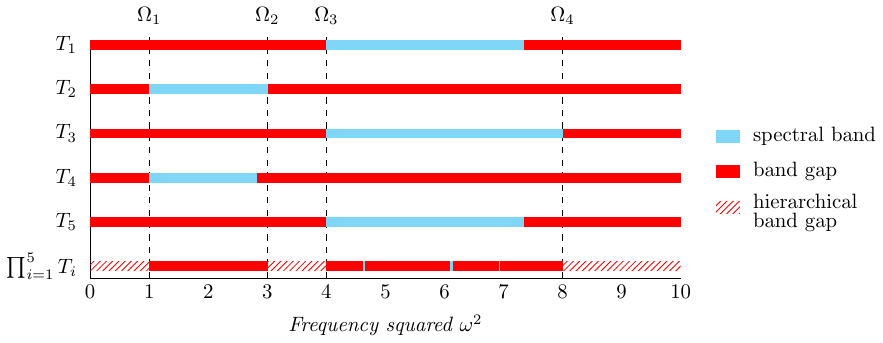}
    \caption{Hierarchical band gaps of a modulated array of coupled pendulums, again with a simple periodic unit cell consisting of five elements. Here, the modulation means that there is a mid-frequency hierarchical band gap in addition to the high- and low-frequency hierarchical band gaps (\emph{cf.} Figure~\ref{fig:pendulums_2}). The hierarchical band gaps are $0 < \omega^2 < \Omega_1$, $\Omega_2 < \omega^2 < \Omega_3$ and $\Omega_4 < \omega^2$, as specified in \eqref{eq:Omegai}.}
    \label{fig:pendulums_3}
\end{figure}

In Figure~\ref{fig:pendulums_3}, we can see that modulating the pendulum length causes the band gaps of the individual systems to be shifted. As a result, there are overlapping band gaps in the centre of the frequency range, meaning that the complex periodic system has a hierarchical band gap here, in addition to those at low and high frequencies, that were observed in Figure~\ref{fig:pendulums_2}. In particular, the hierarchical band gaps occur at $0 < \omega^2 < \Omega_1$, $\Omega_2 < \omega^2 < \Omega_3$ and $\Omega_4 < \omega^2$, where
\begin{equation} \label{eq:Omegai}
\begin{split}
    \Omega_1 = \min_{1\le i \le n} \omega_i^2, \qquad
    \Omega_2 = \min_{1\le i \le n} \frac{4\kappa}{m_i} + \omega_i^2, \\
    \Omega_3 = \max_{1\le i \le n} \omega_i^2, \qquad
    \Omega_4 = \max_{1\le i \le n} \frac{4\kappa}{m_i} + \omega_i^2.
\end{split}
\end{equation}
That $\Omega_2\leq \Omega_3$ in this example is the key property that leads to the hierarchical band gap in the centre of the spectrum. 

This model of coupled pendulums is significant due to its applications in understanding electron dynamics, as outlined in \cite{economou2006quantum}. The arrangement of alternating pendulums was chosen for simplicity of the example and could easily be generalised to a finite number of different lengths to accommodate the variety of problems arising from quantum physics.

\subsection{Resonant phononic crystals}

Lastly we consider periodic systems of coupled acoustic resonators. This could be either a one-dimensional acoustic waveguide containing resonant material inclusions, as studied by \emph{e.g.} \cite{zhao2018topological, leroy2009design} and sketched in Figure~\ref{fig:phononic}, or a coupled array of Helmholtz resonators, as studied by \emph{e.g.} \cite{fang2006ultrasonic, richoux2015disorder, sugimoto1992propagation, zhao2021subwavelength} and sketched in Figure~\ref{fig:helmholtz}. The local resonance in both systems can be modelled using a frequency-dependent effective mass term that is singular at resonance.

\begin{figure}[H]
    \centering
    \begin{subfigure}{\textwidth}
    \centering
    \includegraphics{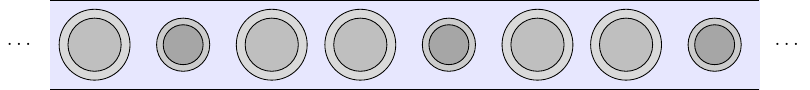}
    \caption{Phononic crystals composed of a soft material (such as air or rubber) surrounded by a much less compressible material (such as water) experience subwavelength local resonances. The resonant frequencies can be modulated by varying the size, shape or material parameters of the inclusions.}
    \label{fig:phononic}
    \end{subfigure}

    \vspace{0.5cm}

    \begin{subfigure}{\textwidth}
    \centering
        \includegraphics{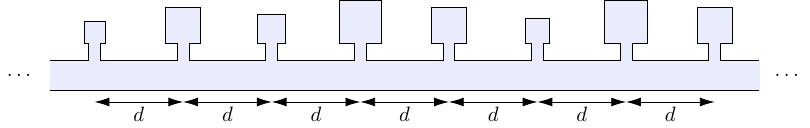}
    \caption{Coupled Helmholtz resonators. The resonant frequencies can be modulated by varying the size and shape of the Helmholtz resonators.}
    \label{fig:helmholtz}
    \end{subfigure}
    \caption{The propagation of acoustic waves in one-dimensional locally resonant phononic crystals can be modelled by the difference equations considered in this work. The local resonances could be due to, for example, Helmholtz resonators or high-contrast material inclusions.}
\end{figure}

The canonical model for wave propagation in locally resonant phononic systems is the difference equation
\begin{equation}\label{eq:resonant}
    \left(2\kappa - M_i^\mathrm{eff}(\omega) \omega^2\right) u_i = (\kappa u_{i-1} + \kappa u_{i+1}),
\end{equation}
where the frequency-dependent effective mass $M_i^\mathrm{eff}(\omega)$ is given by
\begin{equation}
    M_i^\mathrm{eff}(\omega)=M+\frac{m\omega_i^2}{\omega_i^2-\omega^2}.
\end{equation}
Here, $\omega_i$ is the resonant frequency of the locally resonant elements and $m$ and $M$ are constants. This was derived in the appendices of \cite{zhao2018topological}, for example.

An example of the hierarchical band gaps that exist within the discrete system \eqref{eq:resonant} is given in Figure~\ref{fig:phononic_3}. In this example, the effective mass term $M_i^{\text{eff}}$ has been modulated by varying the resonant frequency so that $\omega_1^2 = 6$, $\omega_2^2 = 4$, $\omega_3^2 = 8$, $\omega_4^2 = 6$ and $\omega_5^2 = 10$, with $m=0.5$ and $M=2$ fixed. In this case, all the unit cells have low frequency pass bands, so as a result there is no low frequency band gap. However, the resonances create common band gaps at intermediate frequencies, leading to a hierarchical band gap in the interval $\Omega_1 < \omega^2 < \Omega_2$, as highlighted in Figure~\ref{fig:phononic_3}. The constituent resonant elements also have high-frequency band gaps, that are similarly inherited by the complex structure and predicted by Theorem~\ref{thm:general_result}.

\begin{figure}
    \centering
    \includegraphics{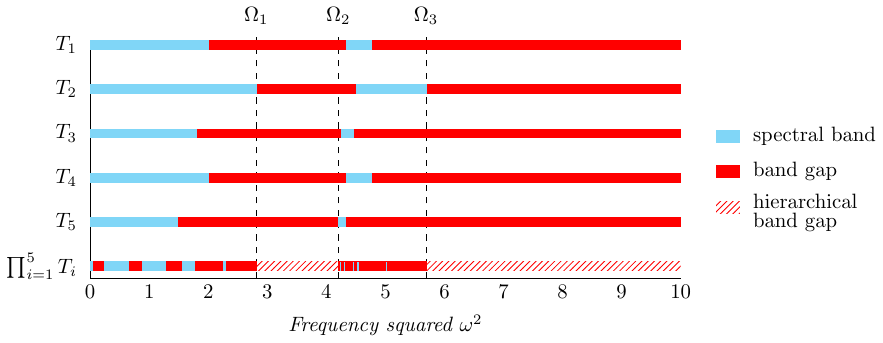}
    \caption{The hierarchical band gaps of a locally resonant phononic crystal with a complex unit cell inherits some of the band gaps from its constituent resonators, including a high-frequency hierarchical band gap. The hierarchical band gaps are $\Omega_1 < \omega^2 < \Omega_2$ and $\Omega_3 < \omega^2$.}
    \label{fig:phononic_3}
\end{figure}

\subsection{Fibonacci tilings}

An example of periodic materials with complex unit cells that have been studied extensively are those based on Fibonacci tilings. That is, given two simple materials A and B, a sequence of materials can be formed through the substitution rule
\begin{equation}
    \text{A}\mapsto\text{AB}, \quad \text{B}\mapsto\text{A}.
\end{equation}
This leads to the sequence of materials A, AB, ABA, ABAAB, ABAABABA, and so on. This sequence earns the name ``Fibonacci'' as each element can be obtained by combining the previous two and the number of elements in each unit cell is given by the corresponding Fibonacci number.

Systems based on the Fibonacci tiling have exotic properties, particularly in the limit of the sequence, which converges to a material that is non-periodic and quasicrystalline \cite{jagannathan2021fibonacci}. Interesting questions concern the system's topological properties \cite{kraus2012topological} and its ability to localise wave energy \cite{kohmoto1983localization}. Materials based on Fibonacci tilings have been used to create waveguides \cite{davies2022symmetry} and realise negative refraction \cite{morini2019negative}. 

An important question for wave systems based on Fibonacci tilings is how the Floquet-Bloch spectrum behaves in the limit of the sequence. It is known that the limiting system typically has a pattern of spectral bands that is a Cantor set \cite{kohmoto1984cantor, sutHo1989singular}. Further, it has been observed that the sequence is such that there are band gaps which emerge early in the sequence and persist for all subsequent iterations (sometimes known as \emph{super band gaps}) \cite{davies2023super, morini2018waves}.

The stability of band gaps in the sequence of Fibonacci tilings can be partially but straightforwardly explained using our theory of hierarchical band gaps. This is demonstrated in Figure~\ref{fig:pendulums_4}, where we consider the system of modulated coupled pendulums from Section~\ref{sec:pendvary}. We take the constituent materials A and B to be the materials with associated matrices $T_1$ and $T_2$ as in Figure~\ref{fig:pendulums_3}. Then,  we consider unit cells given by successive Fibonacci tilings:
\begin{equation}
    F_1 = T_1,\, F_2 = T_2,\, F_3 = T_2 T_1,\, F_4 = T_2 T_1 T_2,\, F_5 = T_2 T_1 T_2 T_2 T_1,\, F_6 = T_2 T_1 T_2 T_2 T_1 T_2 T_1 T_2.
\end{equation}

\begin{figure}[H]
    \centering
    \includegraphics{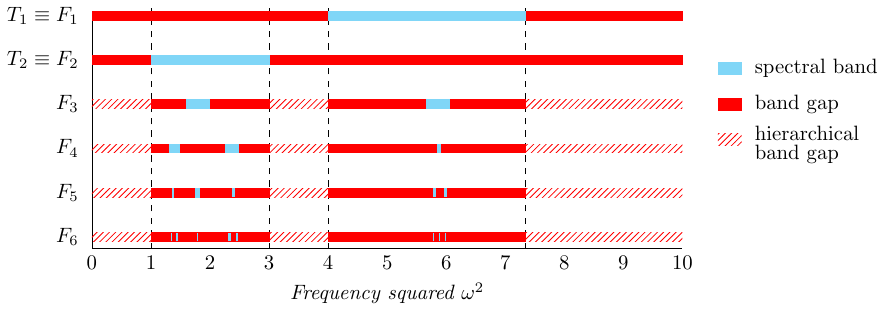}
    \caption{Coupled pendulums arranged in periodic unit cells based on Fibonacci tilings exhibit intricate spectra, some features of which can be predicted using the theory of hierarchical band gaps. The system inherits a high- and low-frequency band-gaps as well as a mid-frequency band gap.}
    \label{fig:pendulums_4}
\end{figure}

It can be seen from Figure~\ref{fig:pendulums_4} that the hierarchical band gaps, predicted by Theorem~\ref{thm:general_result}, persist throughout the sequence of Fibonacci tilings as the unit cell grows and the spectrum becomes increasingly intricate. Clearly, however, Theorem~\ref{thm:general_result} cannot hope to describe the exotic behaviour of the intricate pattern of spectral bands and gaps that emerges outside of these intervals. This behaviour is due to the geometric self similarity of the system (having arbitrarily large sections that repeat infinitely often, even in the limiting non-periodic system) and requires approaches similar to \emph{e.g.} \cite{sutHo1989singular, davies2023super} to handle. Nevertheless, particularly with wave control applications in mind, being able to predict some of the main spectral gaps without needing to invoke such sophisticated machinery is valuable.

\section{Concluding remarks}

This work provides a simple approach to predict band gaps of complex periodic systems by decomposing them into simpler (singly periodic) systems, whose spectra can often be calculated explicitly. This is achieved by re-framing the Floquet-Bloch spectral problem as a problem of matrix algebra. We have proved that, although the set $\M\subset\SL$ of real unimodular two-by-two matrices with trace greater than two is not closed under matrix multiplication, it is possible to write down a subset $\Tc$ of matrices whose products must be in $\M$. This result, stated in Theorem~\ref{thm:general_result}, exploits the symmetries of the subset $\Tc$. This implies that, given a collection of singly periodic systems (each described by a version of the difference equation \eqref{eq:difference}) that has a common band gap, any complex periodic system formed by combining their unit cells must have a (hierarchical) band gap at that frequency range.

Some valuable generalisations of this result would be, firstly, to discrete systems with modulated couplings in addition to the on-site modulation present in \eqref{eq:difference}. We conjecture that it should similarly be possible to formulate a subset of matrices in $\M$ whose elements correspond to these systems and whose products are themselves in $\M$. Further, generalisations to multi-dimensional systems (either discrete or differential) would be valuable but would require a modification of the two-by-two matrix framework used here (either to larger matrix systems or to continuous operators). Finally, it would be valuable to develop a theory capable of handling cases where one (or more) of the matrices in a large product has trace less than two. As is clear from the examples presented in section~\ref{sec:examples}, Theorem~\ref{thm:general_result} yields a sufficient but not necessary condition for a complex periodic system to have a band gap. While more refined approaches exist for specific examples (\emph{e.g.} Fibonacci tilings \cite{sutHo1989singular, davies2023super}), it would be useful to develop approaches in the spirit of the method developed in this work, that reveal fundamental principles for guiding the design of complex periodic materials without the need for expensive computations.

\section*{Declaration of Interests}

The authors do not work for, advise, own shares in, or receive funds from any organisation that could benefit from this article, and have declared no affiliation other than their research organisations.

\section*{Acknowledgement}
The work of B.D. was supported the Engineering and Physical Sciences Research Council (EPSRC) under grant number EP/X027422/1.

\bibliographystyle{abbrv}
\bibliography{references}{}

\begin{thebibliography}{10}

\bibitem{allaire1992homogenization}
G.~Allaire.
\newblock Homogenization and two-scale convergence.
\newblock {\em SIAM J. Math. Anal.}, 23(6):1482--1518, 1992.

\bibitem{bansil2016colloquium}
A.~Bansil, H.~Lin, and T.~Das.
\newblock Colloquium: Topological band theory.
\newblock {\em Rev. Mod. Phys.}, 88(2):021004, 2016.

\bibitem{bao2007surface}
Y.-J. Bao, B.~Zhang, Z.~Wu, J.-W. Si, M.~Wang, R.-W. Peng, X.~Lu, J.~Shao,
  Z.-f. Li, X.-P. Hao, et~al.
\newblock Surface-plasmon-enhanced transmission through metallic film
  perforated with fractal-featured aperture array.
\newblock {\em Appl. Phys. Lett.}, 90(25), 2007.

\bibitem{brillouin1953wave}
L.~N. Brillouin.
\newblock {\em Wave Propagation in Periodic Structures: Electric Filters and
  Crystal Lattices}.
\newblock Dover, New York, 1953.

\bibitem{craster2012acoustic}
R.~V. Craster and S.~Guenneau.
\newblock {\em Acoustic Metamaterials: Negative Refraction, Imaging, Lensing
  and Cloaking}, volume 166 of {\em Springer Series in Materials Science}.
\newblock Springer Science \& Business Media, 2012.

\bibitem{davies2022symmetry}
B.~Davies and R.~V. Craster.
\newblock Symmetry-induced quasicrystalline waveguides.
\newblock {\em Wave Motion}, 115:103068, 2022.

\bibitem{davies2023super}
B.~Davies and L.~Morini.
\newblock Super band gaps and periodic approximants of generalised {F}ibonacci
  tilings.
\newblock {\em arXiv preprint arXiv:2302.10063}, 2023.

\bibitem{economou2006quantum}
E.~N. Economou.
\newblock {\em Green's Functions in Quantum Physics}.
\newblock Springer, Berlin, third edition, 2006.

\bibitem{fang2006ultrasonic}
N.~Fang, D.~Xi, J.~Xu, M.~Ambati, W.~Srituravanich, C.~Sun, and X.~Zhang.
\newblock Ultrasonic metamaterials with negative modulus.
\newblock {\em Nat. Mater.}, 5(6):452--456, 2006.

\bibitem{jagannathan2021fibonacci}
A.~Jagannathan.
\newblock The {F}ibonacci quasicrystal: Case study of hidden dimensions and
  multifractality.
\newblock {\em Rev. Mod. Phys.}, 93(4):045001, 2021.

\bibitem{joannopoulos1995molding}
J.~D. Joannopoulos, R.~D. Meade, and J.~N. Winn.
\newblock {\em Photonic Crystals, Molding the Flow of Light}.
\newblock Princeton University Press, Princeton, NJ, 1995.

\bibitem{kohmoto1983localization}
M.~Kohmoto, L.~P. Kadanoff, and C.~Tang.
\newblock Localization problem in one dimension: Mapping and escape.
\newblock {\em Phys. Rev. Lett.}, 50(23):1870, 1983.

\bibitem{kohmoto1984cantor}
M.~Kohmoto and Y.~Oono.
\newblock Cantor spectrum for an almost periodic {S}chr{\"o}dinger equation and
  a dynamical map.
\newblock {\em Phys. Lett. A}, 102(4):145--148, 1984.

\bibitem{kraus2012topological}
Y.~E. Kraus and O.~Zilberberg.
\newblock Topological equivalence between the {F}ibonacci quasicrystal and the
  {H}arper model.
\newblock {\em Phys. Rev. Lett.}, 109(11):116404, 2012.

\bibitem{kuchment2016}
P.~Kuchment.
\newblock An overview of periodic elliptic operators.
\newblock {\em Bulletin of the American Mathematical Society}, 53(3):323--414,
  2016.

\bibitem{leroy2009design}
V.~Leroy, A.~Bretagne, M.~Fink, H.~Willaime, P.~Tabeling, and A.~Tourin.
\newblock Design and characterization of bubble phononic crystals.
\newblock {\em Appl. Phys. Lett.}, 95(17), 2009.

\bibitem{makwana2013localised}
M.~Makwana and R.~Craster.
\newblock Localised point defect states in asymptotic models of discrete
  lattices.
\newblock {\em Quarterly Journal of Mechanics and Applied Mathematics},
  66(3):289--316, 2013.

\bibitem{makwana2018geometrically}
M.~P. Makwana and R.~V. Craster.
\newblock Geometrically navigating topological plate modes around gentle and
  sharp bends.
\newblock {\em Phys. Rev. B}, 98(18):184105, 2018.

\bibitem{markos2008wave}
P.~Markos and C.~M. Soukoulis.
\newblock {\em Wave Propagation: From Electrons to Photonic Crystals and
  Left-Handed Materials}.
\newblock Princeton University Press, Princeton, New Jersey, 2008.

\bibitem{milton2006cloaking}
G.~W. Milton and N.-A.~P. Nicorovici.
\newblock On the cloaking effects associated with anomalous localized
  resonance.
\newblock {\em Proc. R. Soc. A}, 462(2074):3027--3059, 2006.

\bibitem{morini2019negative}
L.~Morini, Y.~Eyzat, and M.~Gei.
\newblock Negative refraction in quasicrystalline multilayered metamaterials.
\newblock {\em J. Mech. Phys. Solids}, 124:282--298, 2019.

\bibitem{morini2018waves}
L.~Morini and M.~Gei.
\newblock Waves in one-dimensional quasicrystalline structures: dynamical trace
  mapping, scaling and self-similarity of the spectrum.
\newblock {\em J. Mech. Phys. Solids}, 119:83--103, 2018.

\bibitem{richoux2015disorder}
O.~Richoux, A.~Maurel, and V.~Pagneux.
\newblock Disorder persistent transparency within the bandgap of a periodic
  array of acoustic {H}elmholtz resonators.
\newblock {\em J. Appl. Phys.}, 117(10), 2015.

\bibitem{smith2004metamaterials}
D.~R. Smith, J.~B. Pendry, and M.~C. Wiltshire.
\newblock Metamaterials and negative refractive index.
\newblock {\em Science}, 305(5685):788--792, 2004.

\bibitem{sugimoto1992propagation}
N.~Sugimoto.
\newblock Propagation of nonlinear acoustic waves in a tunnel with an array of
  {H}elmholtz resonators.
\newblock {\em J. Fluid Mech.}, 244:55--78, 1992.

\bibitem{sutHo1989singular}
A.~S{\"u}t{\H{o}}.
\newblock Singular continuous spectrum on a {C}antor set of zero {L}ebesgue
  measure for the {F}ibonacci {H}amiltonian.
\newblock {\em J. Stat. Phys.}, 56:525--531, 1989.

\bibitem{wen2005resonant}
W.~Wen, L.~Zhou, B.~Hou, C.~T. Chan, and P.~Sheng.
\newblock Resonant transmission of microwaves through subwavelength fractal
  slits in a metallic plate.
\newblock {\em Phys. Rev. B}, 72(15):153406, 2005.

\bibitem{zhang2019symmetry}
P.~Zhang.
\newblock Symmetry and degeneracy of phonon modes for periodic structures with
  glide symmetry.
\newblock {\em J. Mech. Phys. Solids}, 122:244--261, 2019.

\bibitem{zhao2021subwavelength}
D.~Zhao, X.~Chen, P.~Li, and X.-F. Zhu.
\newblock Subwavelength acoustic energy harvesting via topological interface
  states in {1D} {H}elmholtz resonator arrays.
\newblock {\em AIP Adv.}, 11(1), 2021.

\bibitem{zhao2018topological}
D.~Zhao, M.~Xiao, C.~W. Ling, C.~T. Chan, and K.~H. Fung.
\newblock Topological interface modes in local resonant acoustic systems.
\newblock {\em Phys. Rev. B}, 98(1):014110, 2018.

\end{thebibliography}
\end{document}